\documentclass{amsart}

\usepackage{amssymb}
\usepackage{amsmath}
\usepackage{amsthm}
\usepackage{amsfonts}

\usepackage{a4wide}
\usepackage{longtable}
\usepackage{multirow}
\usepackage{setspace}

\newtheorem{theorem}{Theorem}[section]

\newtheorem{lemma}{Lemma}[section]
\newtheorem{conjecture}{Conjecture}[section]

\theoremstyle{definition}
\newtheorem{definition}{Definition}[section]
\newtheorem{remark}{Remark}[section]

\numberwithin{equation}{section}

\begin{document}

\title{Deep transfers of \(p\)-class tower groups}

\author{Daniel C. Mayer}
\address{Naglergasse 53\\8010 Graz\\Austria}
\email{algebraic.number.theory@algebra.at}
\urladdr{http://www.algebra.at}
\thanks{Research supported by the Austrian Science Fund (FWF): P 26008-N25}

\date{June 30, 2017}

\begin{abstract}
Let \(p\) be a prime.
For any finite \(p\)-group \(G\),
the deep transfers \(T_{H,G^\prime}:\,H/H^\prime\to G^\prime/G^{\prime\prime}\)
from the maximal subgroups \(H\) of index \((G:H)=p\) in \(G\) to the derived subgroup \(G^\prime\)
are introduced as an innovative tool for identifying \(G\) uniquely
by means of the family of kernels \(\varkappa_d(G)=(\ker(T_{H,G^\prime}))_{(G:H)=p}\).
For all finite \(3\)-groups \(G\) of coclass \(\mathrm{cc}(G)=1\),
the family \(\varkappa_d(G)\) is determined explicitly.
The results are applied to the Galois groups \(G=\mathrm{Gal}(F_3^{(\infty)}/F)\)
of the Hilbert \(3\)-class towers of all real quadratic fields \(F=\mathbb{Q}(\sqrt{d})\)
with fundamental discriminants \(d>1\), \(3\)-class group \(\mathrm{Cl}_3(F)\simeq C_3\times C_3\),
and total \(3\)-principalization in each of their four unramified cyclic cubic extensions \(E/F\).
A systematic statistical evaluation is given for the complete range \(1<d<5\cdot 10^6\),
and a few exceptional cases are pointed out for \(1<d<64\cdot 10^6\).
\end{abstract}

\maketitle



\section{Introduction}
\label{s:Intro}
\noindent
The layout of this paper is the following.
Deep transfers of finite \(p\)-groups \(G\), with an assigned prime number \(p\),
are introduced as an innovative supplement to the (usual) shallow transfers
\cite{Ma2}
in \S\
\ref{s:ShallowDeep}.
The family \(\varkappa_d(G)=(\ker(T_{H,G^\prime}))_{(G:H)=p}\)
of the kernels of all deep transfers of \(G\)
is called the \textit{deep transfer kernel type} of \(G\)
and will play a crucial role in this paper.
For all finite \(3\)-groups \(G\) of coclass \(\mathrm{cc}(G)=1\),
the deep transfer kernel type \(\varkappa_d(G)=(\ker(T_{H_i,G^\prime}))_{1\le i\le 4}\)
is determined explicitly with the aid of commutator calculus in \S\
\ref{s:3GroupsDeep}
using a parametrized polycyclic power-commutator presentation of \(G\)
\cite{Bl2,Mi,Ne}.
In the concluding \S\
\ref{s:3ClassTowers},
the orders of the deep transfer kernels are sufficient for
identifying the Galois group \(G_3^\infty{F}:=\mathrm{Gal}(F_3^{(\infty)}/F)\)
of the maximal unramified pro-\(3\) extension
of real quadratic fields \(F=\mathbb{Q}(\sqrt{d})\)
with \(3\)-class group \(\mathrm{Cl}_3(F)\simeq C_3\times C_3\),
and total \(3\)-principalization in each of their four unramified cyclic cubic extensions \(E_1,\ldots,E_4\).



\section{Shallow and deep transfer of \(p\)-groups}
\label{s:ShallowDeep}
\noindent
With an assigned prime number \(p\ge 2\),
let \(G\) be a finite \(p\)-group.
Since our focus in this paper will be on the simplest possible non-trivial situation,
we assume that the abelianization \(G/G^\prime\) of \(G\) is of elementary type \((p,p)\) with rank two.
For applications in number theory, concerning \(p\)-class towers,
the Artin pattern has proved to be a decisive collection of information on \(G\).

\begin{definition}
\label{dfn:Shallow}
The \textit{Artin pattern} \(\mathrm{AP}(G):=(\tau(G),\varkappa(G))\) of \(G\) consists of two families
\begin{equation}
\label{eqn:ArtinPattern}
\tau(G):=(H_i/H_i^\prime)_{1\le i\le p+1} \text{ and } \varkappa(G):=(\ker(T_{G,H_i}))_{1\le i\le p+1}
\end{equation}
containing the targets and kernels
of the Artin transfer homomorphisms \(T_{G,H_i}:\,G/G^\prime\to H_i/H_i^\prime\)
\cite{Ma9}
from \(G\) to its \(p+1\) maximal subgroups \(H_i\) with \(i\in\lbrace 1,\ldots,p+1\rbrace\).
Since the maximal subgroups form the shallow layer \(\mathrm{Lyr}_1(G)\)
of subgroups of index \((G:H_i)=p\) of \(G\),
we shall call the \(T_{G,H_i}\) the \textit{shallow transfers} of \(G\),
and \(\varkappa_s(G):=\varkappa(G)\) the \textit{shallow transfer kernel type} (sTKT) of \(G\).
\end{definition}

\noindent
We recall
\cite{Ma2}
that the sTKT is usually simplified
by a family of non-negative integers, in the following way.
For \(1\le i\le p+1\),
\begin{equation}
\label{eqn:ShallowTKT}
\varkappa_s(G)_i:=
\begin{cases}
j & \text{ if } \ker(T_{G,H_i})=H_j/G^\prime \text{ for some } j\in\lbrace 1,\ldots,p+1\rbrace, \\
0 & \text{ if } \ker(T_{G,H_i})=G/G^\prime.
\end{cases}
\end{equation}

\noindent
The progressive innovation in this paper, however,
is the introduction of the deep Artin transfer.

\begin{definition}
\label{dfn:Deep}
By the \textit{deep transfers}
we understand the Artin transfer homomorphisms \(T_{H_i,G^\prime}:\,H_i/H_i^\prime\to G^\prime/G^{\prime\prime}\)
\cite{Ma9}
from the maximal subgroups \(H_1,\ldots,H_{p+1}\) to the commutator subgroup \(G^\prime\) of \(G\),
which forms the deep layer \(\mathrm{Lyr}_2(G)\)
of the (unique) subgroup of index \((G:G^\prime)=p^2\) of \(G\) with abelian quotient \(G/G^\prime\).
Accordingly, we call the family
\begin{equation}
\label{eqn:DeepTKT}
\varkappa_d(G)=(\#\ker(T_{H_i,G^\prime}))_{1\le i\le p+1}
\end{equation}
the \textit{deep transfer kernel type} (dTKT) of \(G\).
\end{definition}

\noindent
We point out that, as opposed to the sTKT, the members of the dTKT are only \textit{cardinalities},
since this will suffice for reaching our intended goals in this paper.
This preliminary coarse definition is open to further refinement in subsequent publications.
(See the proof of Theorem
\ref{thm:Main}.)



\section{Identification of \(3\)-groups by deep transfers}
\label{s:3GroupsDeep}
\noindent
The drawback of the sTKT is the fact that occasionally
several non-isomorphic \(p\)-groups \(G\) share a common Artin pattern \(\mathrm{AP}(G):=(\tau(G),\varkappa_s(G))\)
\cite[Thm. 7.2, p. 158]{Ma11b}.
The benefit of the dTKT is its ability
to distinguish the members of such batches of \(p\)-groups
which have been inseparable up to now.
After the general introduction of the dTKT for arbitrary \(p\)-groups in \S\
\ref{s:ShallowDeep},
we are now going to demonstrate its advantages
in the particular situation of the prime \(p=3\) and finite \(3\)-groups \(G\) of coclass \(\mathrm{cc}(G)=1\),
which are necessarily metabelian with second derived subgroup \(G^{\prime\prime}=1\)
and abelianization \(G/G^\prime\simeq C_3\times C_3\),
according to Blackburn
\cite{Bl1}.

For the statement of our main theorem,
we need a precise ordering of the four maximal subgroups \(H_1,\ldots,H_4\) of the group \(G=\langle x,y\rangle\),
which can be generated by two elements \(x,y\),
according to the Burnside basis theorem.
For this purpose, we select the generators \(x,y\) such that
\begin{equation}
\label{eqn:MaximalSubgroups}
H_1=\langle y,G^\prime\rangle,\quad
H_2=\langle x,G^\prime\rangle,\quad
H_3=\langle xy,G^\prime\rangle,\quad
H_4=\langle xy^2,G^\prime\rangle,
\end{equation}
and \(H_1=\chi_2(G)\),
provided that \(G\) is of nilpotency class \(\mathrm{cl}(G)\ge 3\).
Here we denote by
\begin{equation}
\label{eqn:TwoStepCentralizer}
\chi_2(G):=\lbrace g\in G\mid (\forall\ h\in G^\prime)\ \lbrack g,h\rbrack\in\gamma_4(G)\rbrace
\end{equation}
the \textit{two-step centralizer} of \(G^\prime\) in \(G\),
where we let \((\gamma_i(G))_{i\ge 1}\) be the lower central series of \(G=:\gamma_1(G)\)
with \(\gamma_i(G)=\lbrack\gamma_{i-1}(G),G\rbrack\) for \(i\ge 2\), in particular, \(\gamma_2(G)=G^\prime\).

The identification of the groups will be achieved with the aid of
parametrized polycyclic power-commutator presentations, as given by
Blackburn
\cite{Bl2},
Miech
\cite{Mi},
and Nebelung
\cite{Ne}:
\begin{equation}
\label{eqn:Presentation}
\begin{aligned}
G_a^n(z,w) := \langle x,y,s_2,\ldots,s_{n-1}\mid s_2=\lbrack y,x\rbrack,\ (\forall_{i=3}^n)\ s_i=\lbrack s_{i-1},x\rbrack,\ s_n=1,\ \lbrack y,s_2\rbrack=s_{n-1}^a, \\
(\forall_{i=3}^{n-1})\ \lbrack y,s_i\rbrack=1,\ x^3=s_{n-1}^w,\ y^3s_2^3s_3=s_{n-1}^z,\ (\forall_{i=2}^{n-3})\ s_i^3s_{i+1}^3s_{i+2}=1,\ s_{n-2}^3=s_{n-1}^3=1\ \rangle,
\end{aligned}
\end{equation}
where \(a\in\lbrace 0,1\rbrace\) and \(w,z\in\lbrace -1,0,1\rbrace\) are bounded parameters,
and the \textit{index of nilpotency} \(n=\mathrm{cl}(G)+1=\mathrm{cl}(G)+\mathrm{cc}(G)=\log_3(\mathrm{ord}(G))=:\mathrm{lo}(G)\) is an unbounded parameter.


\begin{lemma}
\label{lem:Powers}
Let \(G\) be an arbitrary group with elements \(x,y\in G\).
Then the second and third power of the product \(xy\) are given by
\begin{enumerate}
\item
\((xy)^2=x^2y^2s_2t_3\), where \(s_2:=\lbrack y,x\rbrack\), \(t_3:=\lbrack s_2,y\rbrack\),
\item
\((xy)^3=x^3y^3(s_2t_3^2t_4)^2s_3u_4^2u_5s_2t_3\), where \(s_3=\lbrack s_2,x\rbrack\), \(t_4=\lbrack t_3,y\rbrack\),
\(u_4=\lbrack s_3,y\rbrack\), \(u_5=\lbrack u_4,y\rbrack\).
\end{enumerate}
If \(G\simeq G_a^n(z,w)\), then \((xy)^2=x^2y^2s_2s_{n-1}^{-a}\) and \((xy)^3=x^3y^3s_2^3s_3s_{n-1}^{-2a}\),
and the second and third power of \(xy^2\) are given by
\((xy^2)^2=x^2y^4s_2^2s_{n-1}^{-2a}\) and \((xy^2)^3=x^3y^6s_2^6s_3^2s_{n-1}^{-2a}\).
\end{lemma}


\begin{proof}
We prepare the calculation of the powers by proving a few preliminary identities:\\
\(yx=1\cdot yx=xyy^{-1}x^{-1}\cdot yx=xy\cdot y^{-1}x^{-1}yx=xy\cdot\lbrack y,x\rbrack=xys_2\), and similarly\\
\(s_2y=ys_2\cdot\lbrack s_2,y\rbrack=ys_2t_3\) and
\(t_3y=yt_3\cdot\lbrack t_3,y\rbrack=yt_3t_4\) and
\(s_2x=xs_2\cdot\lbrack s_2,x\rbrack=xs_2s_3\) and\\
\(s_3y=ys_3\cdot\lbrack s_3,y\rbrack=ys_3u_4\) and
\(u_4y=yu_4\cdot\lbrack u_4,y\rbrack=yu_4u_5\). Furthermore,\\
\(yx^2=yx\cdot x=xys_2\cdot x=xy\cdot s_2x=xy\cdot xs_2s_3=x\cdot yx\cdot s_2s_3=x\cdot xys_2\cdot s_2s_3=x^2ys_2^2s_3\),\\
\(s_2y^2=s_2y\cdot y=ys_2t_3\cdot y=ys_2\cdot t_3y=ys_2\cdot yt_3t_4=y\cdot s_2y\cdot t_3t_4=y\cdot ys_2t_3\cdot t_3t_4=y^2s_2t_3^2t_4\),\\
\(s_3y^2=s_3y\cdot y=ys_3u_4\cdot y=ys_3\cdot u_4y=ys_3\cdot yu_4u_5=y\cdot s_3y\cdot u_4u_5=y\cdot ys_3u_4\cdot u_4u_5=y^2s_3u_4^2u_5\).\\
Now the second power of \(xy\) is\\
\((xy)^2=xyxy=x\cdot yx\cdot y=x\cdot xys_2\cdot y=x^2y\cdot s_2y=x^2y\cdot ys_2t_3=x^2y^2s_2t_3\)\\
and the third power of \(xy\) is\\
\((xy)^3=xy\cdot (xy)^2=xy\cdot x^2y^2s_2t_3=x\cdot yx^2\cdot y^2s_2t_3=x\cdot x^2ys_2^2s_3\cdot y^2s_2t_3=x^3ys_2^2\cdot s_3y^2\cdot s_2t_3=\)\\
\(=x^3ys_2^2\cdot y^2s_3u_4^2u_5\cdot s_2t_3=x^3ys_2\cdot s_2y^2\cdot s_3u_4^2u_5s_2t_3=x^3ys_2\cdot y^2s_2t_3^2t_4\cdot s_3u_4^2u_5s_2t_3=\)\\
\(=x^3y\cdot s_2y^2\cdot s_2t_3^2t_4s_3u_4^2u_5s_2t_3=x^3y\cdot y^2s_2t_3^2t_4\cdot s_2t_3^2t_4s_3u_4^2u_5s_2t_3=x^3y^3(s_2t_3^2t_4)^2s_3u_4^2u_5s_2t_3\).\\
If \(G\simeq G_a^n(z,w)\), then \(t_4=u_4=u_5=1\), \(t_3=s_{n-1}^{-a}\), \(t_3^3=s_{n-1}^{-3a}=1\), and \(G^\prime\) is abelian.
\end{proof}


\begin{theorem}
\label{thm:Main}
\textbf{(\(3\)-groups \(G\) of coclass \(\mathrm{cc}(G)=1\).)}
Let \(G\) be a finite \(3\)-group of coclass \(\mathrm{cc}(G)=1\)
and order \(\mathrm{ord}(G)=3^n\) with an integer exponent \(n\ge 2\).
Then the shallow and deep transfer kernel type of \(G\) are given
in dependence on the relational parameters \(a,n,w,z\) of \(G\simeq G_a^n(z,w)\)
by Table
\ref{tbl:Main}.
\end{theorem}


\renewcommand{\arraystretch}{1.2}

\begin{table}[ht]
\caption{Shallow and deep TKT of \(3\)-groups \(G\) with \(\mathrm{cc}(G)=1\)}
\label{tbl:Main}
\begin{center}
\begin{tabular}{|l|c|l|c|c|}
\hline
 \(G\simeq\)     & \(n\)          & Type                  & \(\varkappa_s(G)\) & \(\varkappa_d(G)\) \\
\hline
 \(G_0^n(0,0)\)  & \(=2\)         & \(\mathrm{a}.1^\ast\) & \((0,0,0,0)\)      & \((3,3,3,3)\)      \\
 \(G_0^n(0,0)\)  & \(\ge 3\)      & \(\mathrm{a}.1^\ast\) & \((0,0,0,0)\)      & \((9,9,9,9)\)      \\
\hline
 \(G_1^n(0,0)\)  & \(\ge 5\)      & \(\mathrm{a}.1\)      & \((0,0,0,0)\)      & \((3,9,3,3)\)      \\
 \(G_1^n(0,-1)\) & \(\ge 5\)      & \(\mathrm{a}.1\)      & \((0,0,0,0)\)      & \((3,3,9,9)\)      \\
 \(G_1^n(0,1)\)  & \(\ge 5\)      & \(\mathrm{a}.1\)      & \((0,0,0,0)\)      & \((3,3,3,3)\)      \\
\hline
 \(G_0^n(0,1)\)  & \(\ge 4\)      & \(\mathrm{a}.2\)      & \((1,0,0,0)\)      & \((9,3,3,3)\)      \\
\hline
 \(G_0^n(-1,0)\) & \(\ge 4\) even & \(\mathrm{a}.3\)      & \((2,0,0,0)\)      & \((9,9,3,3)\)      \\
 \(G_0^n(1,0)\)  & \(\ge 5\)      & \(\mathrm{a}.3\)      & \((2,0,0,0)\)      & \((9,9,3,3)\)      \\
\hline
 \(G_0^n(1,0)\)  & \(=4\)         & \(\mathrm{a}.3^\ast\) & \((2,0,0,0)\)      & \((27,9,3,3)\)     \\
\hline
 \(G_0^n(0,1)\)  & \(=3\)         & \(\mathrm{A}.1\)      & \((1,1,1,1)\)      & \((9,3,3,3)\)      \\
\hline
\end{tabular}
\end{center}
\end{table}


\begin{proof}
The shallow TKT \(\varkappa_s(G)\) of all \(3\)-groups \(G\) of coclass \(\mathrm{cc}(G)=1\) has been determined in
\cite{Ma2},
where the designations \(\mathrm{a}.n\) of the types were introduced with \(n\in\lbrace 1,2,3\rbrace\).
Here, we indicate a capable mainline vertex of the tree \(\mathcal{T}^1(R)\) with root \(R=C_3\times C_3\)
\cite[Fig. 1--2, pp. 142--143]{Ma11b}
by the type \(\mathrm{a}.1^\ast\) with a trailing asterisk.
As usual, type \(\mathrm{a}.3^\ast\) indicates the unique \(3\)-group \(G\simeq\mathrm{Syl}_3A_9\)
with \(\tau(G)=\lbrack (3,3,3),(3,3)^3\rbrack\).
Now we want to determine the deep TKT \(\varkappa_d(G)\),
using the presentation of \(G\simeq G_a^n(z,w)\) in Formula
\eqref{eqn:Presentation}.
For this purpose, we need expressions for the images of the deep Artin transfers 
\(T_i:=T_{H_i,G^\prime}:\,H_i/H_i^\prime\to G^\prime\),
for each \(1\le i\le 4\). (Observe that \(p=3\) implies \(G^{\prime\prime}=1\) by
\cite{Bl1}.)
Generally, we have to distinguish
\textit{outer} transfers, \(T_i(g\cdot H_i^\prime)=g^3\) if \(g\in H_i\setminus G^\prime\)
\cite[Eqn. (4), p. 470]{Ma2},
and \textit{inner} transfers, \(T_i(g\cdot H_i^\prime)=g^{1+h+h^2}=g^3\cdot\lbrack g,h\rbrack^3\cdot\lbrack\lbrack g,h\rbrack,h\rbrack\)
if \(g\in G^\prime\) and \(h\) is selected in \(H_i\setminus G^\prime\)
\cite[Eqn. (6), p. 486]{Ma2}.

First, we consider the distinguished two-step centralizer \(H_1=\chi_2(G)\) with \(i=1\).
Then \(H_1=\langle y,G^\prime\rangle\) and \(H_1^\prime=1\) if \(a=0\) (\(H_1\) abelian),
but \(H_1^\prime=\gamma_{n-1}(G)=\langle s_{n-1}\rangle\) if \(a=1\) (\(H_1\) non-abelian)
\cite[Eqn. (3), p. 470]{Ma2}.
The outer transfer is determined by \(T_1(y\cdot H_1^\prime)=y^3=s_2^{-3}s_3^{-1}s_{n-1}^z\).
For the inner transfer, we have
\(T_1(s_j\cdot H_1^\prime)=s_j^{1+y+y^2}=s_j^3\cdot\lbrack s_j,y\rbrack^3\cdot\lbrack\lbrack s_j,y\rbrack,y\rbrack=s_j^3\cdot 1^3\cdot\lbrack 1,y\rbrack=s_j^3\)
for all \(j\ge 3\), but \(T_1(s_2\cdot H_1^\prime)=s_2^3\cdot s_{n-1}^{-3a}\cdot\lbrack s_{n-1}^{-a},y\rbrack=s_2^3\) for \(j=2\),
since \(s_{n-1}^{-a}\in\langle s_{n-1}\rangle=\gamma_{n-1}(G)=\zeta_1(G)\) lies in the centre of \(G\).
The first kernel equation \(s_2^{-3}s_3^{-1}s_{n-1}^z=1\) is solvable by
either \(n=3\), where \(z=0\), \(s_3=1\), \(s_2^3=1\),
or \(n=4\), \(z=1\), where \(s_2^3=1\), \(s_{n-1}^z=s_3\).
The second kernel equation \(s_i^3=1\) is solvable by either \(i=n-1\) or \(i=n-2\).
Thus, the deep transfer kernel is given by
\begin{equation}
\label{eqn:dTKT1}
\ker(T_1)=
\begin{cases}
H_1=\langle y,s_2\rangle\simeq C_3\times C_3 \text{ if } n=3\ (G \text{ extra special}), \\
H_1=\langle y,s_2,s_3\rangle\simeq C_3\times C_3\times C_3 \text{ if } n=4,\ z=1\ (G\simeq\mathrm{Syl}_3A_9), \\
\gamma_{n-2}(G)=\langle s_{n-2},s_{n-1}\rangle\simeq C_3\times C_3 \text{ if } n=4,\ z\ne 1 \text{ or } n\ge 5,\ a=0, \\
\gamma_{n-2}(G)/\gamma_{n-1}(G)\simeq\langle s_{n-2}\rangle\simeq C_3 \text{ if } n\ge 5,\ a=1\ (H_1 \text{ non-abelian}).
\end{cases}
\end{equation}

Second, we put \(i=2\).
Then \(H_2=\langle x,G^\prime\rangle\) and \(H_2^\prime=\gamma_3(G)=\langle s_3,\ldots,s_{n-1}\rangle\).
The outer transfer is determined by \(T_2(x\cdot H_2^\prime)=x^3=s_{n-1}^w\).
The inner transfer is given by
\(T_2(s_j\cdot H_2^\prime)=s_j^{1+x+x^2}=s_j^3\cdot\lbrack s_j,x\rbrack^3\cdot\lbrack\lbrack s_j,x\rbrack,x\rbrack=s_j^3s_{j+1}^3s_{j+2}=1\),
for all \(j\ge 2\), independently of \(a,n,w,z\).
Consequently, the deep transfer kernel is given by
\begin{equation}
\label{eqn:dTKT2}
\ker(T_2)=
\begin{cases}
H_2/H_2^\prime=\langle x,s_2,\ldots,s_{n-1}\rangle/\langle s_3,\ldots,s_{n-1}\rangle\simeq\langle x,s_2\rangle\simeq C_3\times C_3\text{ if } w=0, \\
G^\prime/H_2^\prime=\langle s_2,\ldots,s_{n-1}\rangle/\langle s_3,\ldots,s_{n-1}\rangle\simeq\langle s_2\rangle\simeq C_3\text{ if } w=\pm 1.
\end{cases}
\end{equation}

Next, we put \(i=3\).
Then \(H_3=\langle xy,G^\prime\rangle\) and \(H_3^\prime=\gamma_3(G)=\langle s_3,\ldots,s_{n-1}\rangle\).
The outer transfer is determined by \(T_3(xy\cdot H_3^\prime)=(xy)^3=x^3y^3s_2^3s_3s_{n-1}^{-2a}=s_{n-1}^{w+z-2a}\).
For the inner transfer, we have
\(T_3(s_j\cdot H_3^\prime)=s_j^{1+xy+(xy)^2}=s_j^3\cdot\lbrack s_j,xy\rbrack^3\cdot\lbrack\lbrack s_j,xy\rbrack,xy\rbrack=s_j^3s_{j+1}^3s_{j+2}=1\),
for all \(j\ge 3\), independently of \(a,n,w,z\).
The first kernel equation \(s_{n-1}^{w+z-2a}=1\) \(\Longleftrightarrow\) \(w+z-2a\equiv 0\pmod{3}\) is solvable by
either \(a=w=z=0\)
or \(a=1\), \(w=-1\).

Therefore, the deep transfer kernel is given by
\begin{equation}
\label{eqn:dTKT3}
\ker(T_3)=
\begin{cases}
H_3/H_3^\prime\simeq\langle xy,s_2\rangle\simeq C_3\times C_3 \text{ if either } a=w=z=0 \text{ or } a=1,\ w=-1, \\
G^\prime/H_3^\prime\simeq\langle s_2\rangle\simeq C_3 \text{ otherwise }.
\end{cases}
\end{equation}

Finally, we put \(i=4\).
Then \(H_4=\langle xy^2,G^\prime\rangle\) and \(H_4^\prime=\gamma_3(G)=\langle s_3,\ldots,s_{n-1}\rangle\).
The outer transfer is determined by \(T_4(xy^2\cdot H_4^\prime)=(xy^2)^3=x^3y^6s_2^6s_3^2s_{n-1}^{-2a}=s_{n-1}^{w+2z-2a}\).
The inner transfer is given by
\(T_4(s_j\cdot H_4^\prime)=s_j^{1+xy^2+(xy^2)^2}=s_j^3\cdot\lbrack s_j,xy^2\rbrack^3\cdot\lbrack\lbrack s_j,xy^2\rbrack,xy^2\rbrack=s_j^3s_{j+1}^3s_{j+2}=1\),
for all \(j\ge 3\), independently of \(a,n,w,z\).
The first kernel equation \(s_{n-1}^{w+2z-2a}=1\) \(\Longleftrightarrow\) \(w+2z-2a\equiv 0\pmod{3}\) is solvable by
either \(a=w=z=0\)
or \(a=1\), \(w=-1\).

Thus, the deep transfer kernel is given by
\begin{equation}
\label{eqn:dTKT4}
\ker(T_4)=
\begin{cases}
H_4/H_4^\prime\simeq\langle xy^2,s_2\rangle\simeq C_3\times C_3 \text{ if either } a=w=z=0 \text{ or } a=1,\ w=-1, \\
G^\prime/H_4^\prime\simeq\langle s_2\rangle\simeq C_3 \text{ otherwise }.
\end{cases}
\end{equation}

These finer results are summarized in terms of coarser cardinalities in Table
\ref{tbl:Main}.
\end{proof}



\section{Arithmetical application to \(3\)-class tower groups}
\label{s:3ClassTowers}

\subsection{Real quadratic fields}
\label{ss:RealQuadratic}
\noindent
As a final highlight of our progressive innovations,
we come to a number theoretic application of Theorem
\ref{thm:Main},
more precisely, the unambiguous identification of the pro-\(3\) Galois group
\(G_3^\infty{F}=\mathrm{Gal}(F_3^{(\infty)}/F)\)
of the maximal unramified pro-\(3\) extension \(F_3^{(\infty)}\),
that is the Hilbert \(3\)-class field tower,
of certain real quadratic fields \(F=\mathbb{Q}(\sqrt{d})\)
with fundamental discriminant \(d>1\),
\(3\)-class group \(\mathrm{Cl}_3(F)\) of elementary type \((3,3)\),
and shallow transfer kernel type \(\mathrm{a}.1\), \(\varkappa_s(F)=(0,0,0,0)\),
in its \textit{ground state} with \(\tau(F)\sim\lbrack (9,9),(3,3)^3\rbrack\)
or in a higher \textit{excited state} with \(\tau(F)\sim\lbrack (3^e,3^e),(3,3)^3\rbrack\), \(e\ge 3\).

The first field of this kind with \(d=62\,501\) was discovered by Heider and Schmithals in \(1982\)
\cite{HeSm}.
They computed the sTKT \(\varkappa_s(F)=(0,0,0,0)\) with four total \(3\)-principalizations
in the unramified cyclic cubic extensions \(E_i/F\), \(1\le i\le 4\), on a CDC Cyber mainframe.
The fact that \(d=62\,501\) is a \text{triadic irregular} discriminant
(in the sense of Gauss)
with non-cyclic \(3\)-class group \(\mathrm{Cl}_3(F)\simeq C_3\times C_3\)
has been pointed out earlier in \(1936\) by Pall
\cite{Pa}
already.
The second field of this kind with \(d=152\,949\)
was discovered by ourselves in \(1991\) by computing \(\varkappa_s(F)\) on an AMDAHL mainframe
\cite{Ma04}.
In \(2006\), there followed \(d=252\,977\) and \(d=358\,285\), and many other cases in \(2009\)
\cite{Ma1,Ma3}.

Generally, there are three contestants for the group \(G=G_3^\infty{F}\),
for any assigned state \(\tau(F)\sim\lbrack (3^e,3^e),(3,3)^3\rbrack\), \(e\ge 2\),
and the following \textbf{Main Theorem} admits their identification by means of the deep transfer kernel type.
(See their statistical distribution at the end of section
\ref{ss:RealQuadratic}.)


\begin{theorem}
\label{thm:ArithmeticMain}
\textbf{(\(3\)-class tower groups \(G\) of coclass \(\mathrm{cc}(G)=1\) and type \(\mathrm{a}.1\).)}
Let \(F=\mathbb{Q}(\sqrt{d})\) be a real quadratic field
with fundamental discriminant \(d>1\),
\(3\)-class group \(\mathrm{Cl}_3(F)\simeq C_3\times C_3\),
and shallow transfer kernel type \(\mathrm{a}.1\), \(\varkappa_s(F)=(0,0,0,0)\).
Suppose the state of type \(\mathrm{a}.1\) is determined by the transfer target type
\(\tau(F)\sim\lbrack (3^e,3^e),(3,3)^3\rbrack\) with \(e\ge 2\).

Then the relational parameters \(n\ge 5\) and \(w\in\lbrace -1,0,1\rbrace\)
of the \(3\)-class tower group \(G=G_3^\infty{F}\simeq G_1^n(0,w)\) of \(F\)
are given in dependence on the deep transfer kernel type \(\varkappa_d(F)\) as follows:
\begin{equation}
\label{eqn:ArithmeticMain}
\begin{aligned}
& G\simeq G_1^{2(e+1)}(0,0)  & \text{ with } & n=2(e+1),\ w=0  & \Longleftrightarrow & \quad\varkappa_d(F)\sim (3,9,3,3), \\
& G\simeq G_1^{2(e+1)}(0,-1) & \text{ with } & n=2(e+1),\ w=-1 & \Longleftrightarrow & \quad\varkappa_d(F)\sim (3,3,9,9), \\
& G\simeq G_1^{2(e+1)}(0,1)  & \text{ with } & n=2(e+1),\ w=1  & \Longleftrightarrow & \quad\varkappa_d(F)\sim (3,3,3,3).
\end{aligned}
\end{equation}
\end{theorem}

\begin{proof}
The claim is an immediate consequence of Theorem
\ref{thm:Main}.
\end{proof}


Table
\ref{tbl:GroundState}
shows that the \textit{ground state} \(\tau(F)=\lbrack (9,9),(3,3)^3\rbrack\)
of the sTKT \(\varkappa_s(F)=(0,0,0,0)\) has the nice property that
the smallest three discriminants already realize three different \(3\)-class tower groups
\(G=G_3^\infty{F}\simeq\langle 729,i\rangle\) with \(i\in\lbrace 99,100,101\rbrace\),
identified by their dTKT \(\varkappa_d(F)=\varkappa_d(G)\).

In Table
\ref{tbl:ExcitedState1},
we see that the \textit{first excited state} \(\tau(F)=\lbrack (27,27),(3,3)^3\rbrack\)
of the sTKT \(\varkappa_s(F)=(0,0,0,0)\) does not behave so well:
although the smallest two discriminants
\cite{Ma1,Ma3,Ma14,Ma14b}
already realize two different \(3\)-class tower groups
\(G=G_3^\infty{F}\simeq\langle 6561,i\rangle\) with \(i\in\lbrace 2225,2227\rbrace\),
we have to wait for the seventh occurrence until \(\langle 6561,2226\rangle\) is realized,
as the dTKT \(\varkappa_d(F)=\varkappa_d(G)\) shows.
The counter \(7\) is a typical example of a \textit{statistic delay}.

In both tables, the shortcut MD means the \textit{minimal discriminant}
\cite[Dfn. 6.2, p. 148]{Ma11b}.


\renewcommand{\arraystretch}{1.0}

\begin{table}[ht]
\caption{Deep TKT of \(3\)-class tower groups \(G\) with \(\tau(G)=\lbrack (9,9),(3,3)^3\rbrack\)}
\label{tbl:GroundState}
\begin{center}
\begin{tabular}{|l|c|r|}
\hline
 \(G\)                                        & \(\varkappa_d(G)\)  & MD           \\
\hline
 \(\langle 729,99\rangle\simeq G_1^6(0,0)\)   & \((3,9,3,3)\)       &  \(62\,501\) \\
 \(\langle 729,100\rangle\simeq G_1^6(0,-1)\) & \((3,3,9,9)\)       & \(152\,949\) \\
 \(\langle 729,101\rangle\simeq G_1^6(0,1)\)  & \((3,3,3,3)\)       & \(252\,977\) \\
\hline
\end{tabular}
\end{center}
\end{table}


\renewcommand{\arraystretch}{1.1}

\begin{table}[ht]
\caption{Deep TKT of \(3\)-class tower groups \(G\) with \(\tau(G)=\lbrack (27,27),(3,3)^3\rbrack\)}
\label{tbl:ExcitedState1}
\begin{center}
\begin{tabular}{|l|c|r|l|}
\hline
 \(G\)                                          & \(\varkappa_d(G)\)  & MD               & further discriminants \\
\hline
 \(\langle 6561,2225\rangle\simeq G_1^8(0,0)\)  & \((3,9,3,3)\)       & \(10\,399\,596\) & \(16\,613\,448\) \\
 \(\langle 6561,2226\rangle\simeq G_1^8(0,-1)\) & \((3,3,9,9)\)       & \(27\,780\,297\) & \\
 \(\langle 6561,2227\rangle\simeq G_1^8(0,1)\)  & \((3,3,3,3)\)       &  \(2\,905\,160\) & \(14\,369\,932\), \(15\,019\,617\), \(21\,050\,241\) \\
\hline
\end{tabular}
\end{center}
\end{table}




\begin{figure}[ht]
\caption{Distribution of minimal discriminants for \(G_3^\infty{F}\) on the coclass tree \(\mathcal{T}^1(\langle 9,2\rangle)\)}
\label{fig:MinDscTreeCc1}

\input{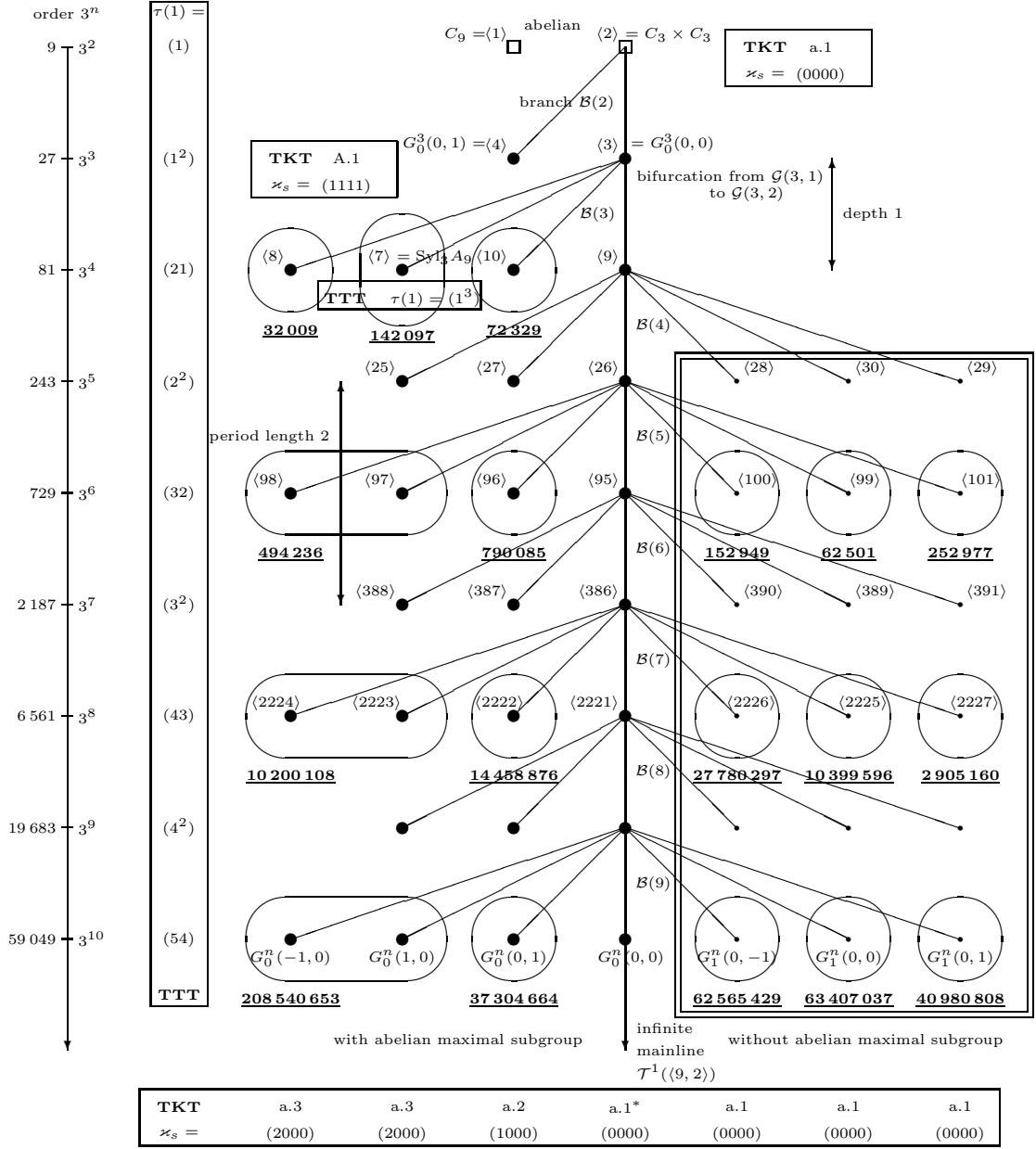}

\end{figure}


The diagram in Figure
\ref{fig:MinDscTreeCc1}
visualizes the initial eight branches of the coclass tree \(\mathcal{T}^1(R)\)
with abelian root \(R=\langle 9,2\rangle\simeq C_3\times C_3\).
Basic definitions, facts, and notation concerning general descendant trees of finite \(p\)-groups
are summarized briefly in
\cite[\S\ 2, pp. 410--411]{Ma4},
\cite{Ma4a}.
They are discussed thoroughly in the broadest detail in the initial sections of
\cite{Ma6}.
Descendant trees are crucial for recent progress in the theory of \(p\)-class field towers
\cite{Ma13,Ma15,Ma15b},
in particular for describing the mutual location of
the second \(p\)-class group \(\mathrm{G}_p^2{F}\) and the \(p\)-class tower group \(\mathrm{G}_p^\infty{F}\)
of a number field \(F\).
Generally, the vertices of the coclass tree in the figure
represent isomorphism classes of finite \(3\)-groups.
Two vertices are connected by a directed edge \(G\to H\) if
\(H\) is isomorphic to the last lower central quotient \(G/\gamma_c(G)\),
where \(c=\mathrm{cl}(G)=n-1\) denotes the nilpotency class of \(G\),
and \(\lvert G\rvert=3\lvert H\rvert\), that is,
\(\gamma_c(G)\simeq C_3\) is cyclic of order \(3\).
See also
\cite[\S\ 2.2, p. 410--411]{Ma4}
and
\cite[\S\ 4, p. 163--164]{Ma6}.

The vertices of the tree diagram in Figure
\ref{fig:MinDscTreeCc1}
are classified by using various symbols:
\begin{enumerate}
\item
big contour squares \(\square\) represent abelian groups,
\item
big full discs {\Large \(\bullet\)} represent metabelian groups
with at least one abelian maximal subgroup,
\item
small full discs {\footnotesize \(\bullet\)} represent metabelian groups
without abelian maximal subgroups.
\end{enumerate}
The groups of particular importance are labelled by a number in angles,
which is the identifier in the SmallGroups Library
\cite{BEO}
of MAGMA
\cite{MAGMA}.
We omit the orders, which are given on the left hand scale.
The sTKT \(\varkappa_s\)
\cite[Thm. 2.5, Tbl. 6--7]{Ma2},
in the bottom rectangle concerns all
vertices located vertically above.
The first component \(\tau(1)\) of the TTT
\cite[Dfn. 3.3, p. 288]{Ma7}
in the left rectangle
concerns vertices \(G\) on the same horizontal level containing an abelian maximal subgroup.
It is given in logarithmic notation.
The periodicity with length \(2\) of branches,
\(\mathcal{B}(j)\simeq\mathcal{B}(j+2)\) for \(j\ge 4\),
sets in with branch \(\mathcal{B}(4)\),
having a root of order \(3^4\).

\(3\)-class tower groups \(G=G_3^\infty{F}\) with coclass \(\mathrm{cc}(G)=1\)
of real quadratic fields \(F=\mathbb{Q}(\sqrt{d})\)
are located as arithmetically realized vertices on the tree diagram in Figure
\ref{fig:MinDscTreeCc1}.
The minimal fundamental discriminants \(d\), i.e. the MDs,
are indicated by underlined boldface integers
adjacent to the oval surrounding the realized vertex
\cite{BEO,MAGMA,Ma9}.

The double contour rectangle surrounds the vertices
which became distinguishable by the progressive innovations in the present paper
and were inseparable up to now.


\renewcommand{\arraystretch}{1.1}

\begin{table}[ht]
\caption{Statistics of \(3\)-class tower groups \(G\) with \(\tau(G)=\lbrack (9,9),(3,3)^3\rbrack\)}
\label{tbl:Statistics}
\begin{center}
\begin{tabular}{|r|r|l||r|r|l|}
\hline
    No. &           \(d\) & \(G\)                      &    No. &           \(d\) & \(G\)                      \\
\hline
  \(1\) &     \(62\,501\) & \(\langle 729,99\rangle\)  & \(35\) & \(2\,710\,072\) & \(\langle 729,100\rangle\) \\
  \(2\) &    \(152\,949\) & \(\langle 729,100\rangle\) & \(36\) & \(2\,851\,877\) & \(\langle 729,99\rangle\)  \\
  \(3\) &    \(252\,977\) & \(\langle 729,101\rangle\) & \(37\) & \(2\,954\,929\) & \(\langle 729,99\rangle\)  \\
  \(4\) &    \(358\,285\) & \(\langle 729,101\rangle\) & \(38\) & \(3\,005\,369\) & \(\langle 729,101\rangle\) \\
  \(5\) &    \(531\,437\) & \(\langle 729,99\rangle\)  & \(39\) & \(3\,197\,864\) & \(\langle 729,100\rangle\) \\
  \(6\) &    \(586\,760\) & \(\langle 729,101\rangle\) & \(40\) & \(3\,197\,944\) & \(\langle 729,101\rangle\) \\
  \(7\) &    \(595\,009\) & \(\langle 729,100\rangle\) & \(41\) & \(3\,258\,120\) & \(\langle 729,101\rangle\) \\
  \(8\) &    \(726\,933\) & \(\langle 729,99\rangle\)  & \(42\) & \(3\,323\,065\) & \(\langle 729,99\rangle\)  \\
  \(9\) &    \(801\,368\) & \(\langle 729,100\rangle\) & \(43\) & \(3\,342\,785\) & \(\langle 729,99\rangle\)  \\
 \(10\) &    \(940\,593\) & \(\langle 729,100\rangle\) & \(44\) & \(3\,644\,357\) & \(\langle 729,99\rangle\)  \\
 \(11\) &    \(966\,489\) & \(\langle 729,99\rangle\)  & \(45\) & \(3\,658\,421\) & \(\langle 729,100\rangle\) \\
 \(12\) & \(1\,177\,036\) & \(\langle 729,99\rangle\)  & \(46\) & \(3\,692\,717\) & \(\langle 729,99\rangle\)  \\
 \(13\) & \(1\,192\,780\) & \(\langle 729,101\rangle\) & \(47\) & \(3\,721\,565\) & \(\langle 729,99\rangle\)  \\
 \(14\) & \(1\,313\,292\) & \(\langle 729,99\rangle\)  & \(48\) & \(3\,799\,597\) & \(\langle 729,100\rangle\) \\
 \(15\) & \(1\,315\,640\) & \(\langle 729,99\rangle\)  & \(49\) & \(3\,821\,244\) & \(\langle 729,99\rangle\)  \\
 \(16\) & \(1\,358\,556\) & \(\langle 729,100\rangle\) & \(50\) & \(3\,869\,909\) & \(\langle 729,99\rangle\)  \\
 \(17\) & \(1\,398\,829\) & \(\langle 729,101\rangle\) & \(51\) & \(3\,995\,004\) & \(\langle 729,101\rangle\) \\
 \(18\) & \(1\,463\,729\) & \(\langle 729,101\rangle\) & \(52\) & \(4\,045\,265\) & \(\langle 729,101\rangle\) \\
 \(19\) & \(1\,580\,709\) & \(\langle 729,100\rangle\) & \(53\) & \(4\,183\,205\) & \(\langle 729,100\rangle\) \\
 \(20\) & \(1\,595\,669\) & \(\langle 729,100\rangle\) & \(54\) & \(4\,196\,840\) & \(\langle 729,100\rangle\) \\
 \(21\) & \(1\,722\,344\) & \(\langle 729,99\rangle\)  & \(55\) & \(4\,199\,901\) & \(\langle 729,101\rangle\) \\
 \(22\) & \(1\,751\,909\) & \(\langle 729,101\rangle\) & \(56\) & \(4\,220\,977\) & \(\langle 729,100\rangle\) \\
 \(23\) & \(1\,831\,097\) & \(\langle 729,99\rangle\)  & \(57\) & \(4\,233\,608\) & \(\langle 729,99\rangle\)  \\
 \(24\) & \(1\,942\,385\) & \(\langle 729,101\rangle\) & \(58\) & \(4\,252\,837\) & \(\langle 729,100\rangle\) \\
 \(25\) & \(2\,021\,608\) & \(\langle 729,99\rangle\)  & \(59\) & \(4\,409\,313\) & \(\langle 729,100\rangle\) \\
 \(26\) & \(2\,042\,149\) & \(\langle 729,101\rangle\) & \(60\) & \(4\,429\,612\) & \(\langle 729,101\rangle\) \\
 \(27\) & \(2\,076\,485\) & \(\langle 729,99\rangle\)  & \(61\) & \(4\,533\,032\) & \(\langle 729,99\rangle\)  \\
 \(28\) & \(2\,185\,465\) & \(\langle 729,101\rangle\) & \(62\) & \(4\,586\,797\) & \(\langle 729,100\rangle\) \\
 \(29\) & \(2\,197\,669\) & \(\langle 729,101\rangle\) & \(63\) & \(4\,662\,917\) & \(\langle 729,100\rangle\) \\
 \(30\) & \(2\,314\,789\) & \(\langle 729,99\rangle\)  & \(64\) & \(4\,680\,701\) & \(\langle 729,99\rangle\)  \\
 \(31\) & \(2\,409\,853\) & \(\langle 729,99\rangle\)  & \(65\) & \(4\,766\,309\) & \(\langle 729,99\rangle\)  \\
 \(32\) & \(2\,433\,221\) & \(\langle 729,101\rangle\) & \(66\) & \(4\,782\,664\) & \(\langle 729,99\rangle\)  \\
 \(33\) & \(2\,539\,129\) & \(\langle 729,101\rangle\) & \(67\) & \(4\,783\,697\) & \(\langle 729,100\rangle\) \\
 \(34\) & \(2\,555\,249\) & \(\langle 729,100\rangle\) & \(68\) & \(4\,965\,009\) & \(\langle 729,?\rangle\) \\
\hline
\end{tabular}
\end{center}
\end{table}


In Table
\ref{tbl:Statistics},
we give the isomorphism type of the \(3\)-class tower group \(G=G_3^\infty{F}\)
of all real quadratic fields \(F=\mathbb{Q}(\sqrt{d})\)
with \(3\)-class group \(\mathrm{Cl}_3(F)\simeq C_3\times C_3\)
and shallow transfer kernel type \(\mathrm{a}.1\), \(\varkappa_s=(0,0,0,0)\),
in its \textit{ground state} \(\tau(F)=\lbrack (9,9),(3,3)^3\rbrack\),
for the complete range \(1<d<5\cdot 10^6\) of \(68\) fundamental discriminants \(d\).
It was determined by means of Theorem
\ref{thm:ArithmeticMain},
applied to the results of computing the (restricted) deep transfer kernel type
\(\varkappa_d(F)=(\#\ker(T_{F_3^{(1)}/E_i}))_{2\le i\le 4}\),
consisting of the orders of the \(3\)-principalization kernels
of those unramified cyclic cubic extensions \(E_i\), \(2\le i\le 4\),
in the Hilbert \(3\)-class field \(F_3^{(1)}\) of \(F\)
whose \(3\)-class group \(\mathrm{Cl}_3(E_i)\) is of type \((3,3)\).
These trailing three components of the TTT \(\tau(F)=\lbrack (9,9),(3,3)^3\rbrack\)
were called its \textit{stable part} in
\cite[Dfn. 5.5, p. 84]{Ma9}.
The computations were done with the aid of the computational algebra system MAGMA
\cite{MAGMA}.
The \(3\)-principalization kernel of the
remaining extension \(E_1\) with \(3\)-class group \(\mathrm{Cl}_3(E_1)\) of type \((9,9)\)
does not contain essential information and can be omitted.
This leading component of the TTT \(\tau(F)=\lbrack (9,9),(3,3)^3\rbrack\)
was called its \textit{polarized part} in
\cite[Dfn. 5.5, p. 84]{Ma9}.
For more details on the concepts \textit{stabilization} and \textit{polarization}, see
\cite[\S\ 6, pp. 90--95]{Ma9}.
A systematic statistical evaluation of Table 
\ref{tbl:Statistics}
shows that, with respect to the complete range \(1<d<5\cdot 10^6\),
the group \(G\simeq\langle 729,99\rangle\) occurs with slightly elevated relative frequency \(40\%\),
whereas \(G\simeq\langle 729,100\rangle\) and \(G\simeq\langle 729,101\rangle\) share the common percentage \(30\%\),
although the automorphism group \(\mathrm{Aut}(G)\) of all three groups has the same order.
However, the proportion \(40:30:30\) for the bound \(5\cdot 10^6\) is obviously not settled yet,
because there are remarkable fluctuations
\(36:36:27\) for \(10^6\),
\(38:29:33\) for \(2\cdot 10^6\),
\(41:24:35\) for \(3\cdot 10^6\), and
\(43:24:33\) for \(4\cdot 10^6\).
We expect an asymptotic limit \(33:33:33\) of the proportion for \(d\to\infty\).


\subsection{Totally real dihedral fields}
\label{ss:TotallyRealDihedral}
\noindent
In fact, we have computed much more information with MAGMA
than mentioned at the end of the previous section
\ref{ss:RealQuadratic}.
To understand the actual scope of our numerical results
it is necessary to recall that each unramified cyclic cubic relative extension \(E_i/F\), \(1\le i\le 4\),
gives rise to a dihedral absolute extension \(E_i/\mathbb{Q}\) of degree \(6\),
that is an \(S_3\)-extension
\cite[Prp. 4.1, p. 482]{Ma1}.
For the trailing three fields \(E_i\), \(2\le i\le 4\),
in the stable part of the TTT \(\tau(F)=\lbrack (9,9),(3,3)^3\rbrack\),
i.e. with \(\mathrm{Cl}_3(E_i)\) of type \((3,3)\),
we have constructed the unramified cyclic cubic extensions \(\tilde{E}_{i,j}/E_i\), \(1\le j\le 4\),
and determined the Artin pattern \(\mathrm{AP}(E_i)\) of \(E_i\),
in particular, the \(3\)-principalization type of \(E_i\) in the fields \(\tilde{E}_{i,j}\).
The dihedral fields \(E_i\) of degree \(6\) share a common polarization
\(\tilde{E}_{i,1}=F_3^{(1)}\), the Hilbert \(3\)-class field of \(F\),
which is contained in the relative \(3\)-genus field \((E_i/F)^\ast\),
whereas the other extensions \(\tilde{E}_{i,j}\) with \(2\le j\le 4\)
are non-abelian over \(F\), for each \(2\le i\le 4\).
Our computational results suggest the following conjecture
concerning the infinite family of totally real dihedral fields
\(E_i\) for varying real quadratic fields \(F\).

\begin{conjecture}
\label{cnj:Dihedral}
\textbf{(\(3\)-class tower groups \(\mathcal{G}\) of totally real dihedral fields.)}
Let \(F=\mathbb{Q}(\sqrt{d})\) be a real quadratic field
with fundamental discriminant \(d>1\),
\(3\)-class group \(\mathrm{Cl}_3(F)\simeq C_3\times C_3\),
and shallow transfer kernel type \(\mathrm{a}.1\), \(\varkappa_s(F)=(0,0,0,0)\),
in the ground state with transfer target type \(\tau(F)\sim\lbrack (9,9),(3,3)^3\rbrack\).
Let \(E_2,E_3,E_4\) be the three unramified cyclic cubic relative extensions of \(F\)
with \(3\)-class group \(\mathrm{Cl}_3(E_i)\) of type \((3,3)\).

Then \(E_i/\mathbb{Q}\) is a totally real dihedral extension of degree \(6\),
for each \(2\le i\le 4\),
and the connection between the component \(\varkappa_d(F)_i=\#\ker(T_{F_3^{(1)}/E_i})\)
of the deep transfer kernel type \(\varkappa_d(F)\) of \(F\)
and the \(3\)-class tower group \(\mathcal{G}_i=G_3^\infty{E_i}=\mathrm{Gal}((E_i)_3^{(\infty)}/E_i)\) of \(E_i\)
is given in the following way:
\begin{equation}
\label{eqn:Dihedral}
\begin{aligned}
 & \varkappa_d(F)_i=3 & \Longleftrightarrow & \quad\mathcal{G}_i\simeq\langle 243,27\rangle & \text{ with } \varkappa_s(\mathcal{G}_i)=(1,0,0,0), \\
 & \varkappa_d(F)_i=9 & \Longleftrightarrow & \quad\mathcal{G}_i\simeq\langle 243,26\rangle & \text{ with } \varkappa_s(\mathcal{G}_i)=(0,0,0,0).
\end{aligned}
\end{equation}
\end{conjecture}

\begin{remark}
\label{rmk:Dihedral}
The conjecture is supported by all \(3\cdot 68=204\) totally real dihedral fields \(E_i\)
which were involved in the computation of Table
\ref{tbl:Statistics}.
A provable argument for the truth of the conjecture is the fact that
\(\varkappa_d(F)_i=\#\ker(T_{F_3^{(1)}/E_i})=\#\varkappa_s(E_i)_1=\#\varkappa_s(\mathcal{G}_i)_1\),
for \(2\le i\le 4\),
but it does not explain why the sTKT \(\varkappa_s(\mathcal{G}_i)\) is \(\mathrm{a}.2\) with a fixed point if \(\varkappa_d(F)_i=3\).
It is interesting that a dihedral field \(E_i\) of degree \(6\) is satisfied with a non-\(\sigma\) group,
such as \(\langle 243,27\rangle\), as its \(3\)-class tower group.
On the other hand, it is not surprising that a mainline group,
such as \(\langle 243,26\rangle\) with sTKT \(\mathrm{a}.1^\ast\) and relation rank \(d_2=4\), is possible as \(\mathcal{G}_i=G_3^\infty{E_i}\),
since the upper Shafarevich bound for the relation rank of the \(3\)-class tower group
of a totally real dihedral field \(E_i\) of degree \(6\) with \(\mathrm{Cl}_3(E_i)\simeq C_3\times C_3\)
is given by \(\varrho+r_1+r_2-1=2+6+0-1=7>4\)
\cite[Thm. 1.3, p. 75]{Ma15}.
\end{remark}

\noindent
Assuming an asymptotic limit \(33:33:33\) of the proportion of the real quadratic \(3\)-class tower groups
\(G\in\lbrace\langle 729,99\rangle,\langle 729,100\rangle,\langle 729,101\rangle\rbrace\)
for the ground state of sTKT \(\mathrm{a}.1\),
we can also conjecture an asymptotic limit \(33:66\) of the corresponding totally real dihedral \(3\)-class tower groups
\(\mathcal{G}_i\in\lbrace\langle 243,26\rangle,\lbrace\langle 243,27\rangle\rbrace\),
since the restricted dTKTs \((9,3,3)\), \((3,9,9)\), \((3,3,3)\) together contain
three times the \(9\) and six times the \(3\) in Equation
\eqref{eqn:Dihedral}.



\section{Acknowledgements}
\label{s:Acknowledgements}

\noindent
The author gratefully acknowledges that his research was supported
by the Austrian Science Fund (FWF): P 26008-N25.




\end{document}